\documentclass[12pt]{article}
\usepackage{mathrsfs}
\usepackage{amsmath}
\usepackage{amsmath,amsthm,amssymb,amscd}
\usepackage{latexsym}
\usepackage[colorlinks,linkcolor=blue,anchorcolor=blue,citecolor=blue,CJKbookmarks=True]{hyperref}
\usepackage[numbers,sort&compress]{natbib}
\usepackage{cases}
\usepackage{makecell}
\usepackage{multirow}
\usepackage{geometry}
\usepackage{tabularx}
\usepackage{graphicx}
\usepackage{tikz}
\usepackage{tikz-cd}

\allowdisplaybreaks
\newtheorem{theorem}{Theorem}[section]
\newtheorem{corollary}[theorem]{Corollary}
\newtheorem{lemma}[theorem]{Lemma}
\newtheorem{example}[theorem]{Example}
\newtheorem{proposition}[theorem]{Proposition}
\newtheorem{question}[theorem]{Question}
\theoremstyle{definition}
\newtheorem{definition}[theorem]{Definition}
\newtheorem{remark}[theorem]{Remark}

\numberwithin{equation}{section}

\begin{document}
\begin{center}
{\large  \bf On $\ast$-Reversible and Generalized $\ast$-Reversible Rings}\\
\vspace{0.8cm}   Huaxi Chen$^{a}$,  Long Wang$^{b}$and Honglin Zou$^{c}$\footnote{Corresponding author. Email: honglinzou@sina.com or lwangmath@yzu.edu.cn} \\
\vspace{0.5cm} {\small $^{a}$School of Mathematical Sciences, Bengbu University, Bengbu, China\\
\small $^{b}$School of Mathematical Sciences, Yangzhou University, Yangzhou,  China\\
\small $^{c}$College of Basic Science, Zhejiang Shuren University, Hangzhou, P. R. China}
\end{center}

\bigskip

{ \bf  Abstract:}  \leftskip0truemm\rightskip0truemm
Let $R$ be a $\ast$-ring with $a,b\in R$. A ring $R$ is said to be $\ast$-reversible if $ab=0$ implies $b^{\ast}a=0$.
In this paper, we first establish several new characterizations of $\ast$-reversible rings and reversible rings.
In particular, we prove that $\ast$-reversible rings coincide with $\ast$-symmetric rings.
Using these characterizations, we introduce two new classes of generalized $\ast$-reversible rings: pro-$\ast$-reversible rings and nil-$\ast$-reversible rings.
A ring $R$ is called pro-$\ast$-reversible if $ab\in P(R)$ implies $b^{\ast}a\in P(R)$,
and $R$ is nil-$\ast$-reversible if for every $c\in N(R)$, $cb=0$ yields both $b^{\ast}c=0$ and $cb^{\ast}=0$.
The basic properties and characterizations of pro-$\ast$-reversible and nil-$\ast$-reversible rings are investigated.
The interrelationships among all these ring classes are considered. The related examples to distinguish these rings are constructed.
\\{  \textbf{Keywords:}} reversible rings; $\ast$-reversible rings; $\ast$-symmetric rings; pro-$\ast$-reversible rings;
nil-$\ast$-reversible rings.
\\\noindent { \textbf{2010 Mathematics Subject Classification:}} 16W10, 16U80, 16N40.
 \bigskip

\section{Introduction}

Reversible rings and their generalizations constitute an important class of rings in noncommutative ring theory.
Recall that a ring $R$ is called reversible if $ab=0$ for all $a,b\in R$ implies $ba=0$.
This notion generalizes commutative rings and reduced rings, and its structural properties,
characterizations and related extensions have attracted extensive attention from algebra scholars over the past decades.
Lambek \cite{L1} introduced symmetric rings as another generalization of reduced rings,
and Marks \cite{M1} systematically studied the basic properties and interrelationships between symmetric rings and reversible rings. Later, numerous generalized symmetric and reversible rings, such as $e$-symmetric rings, weak symmetric rings and
$(g,e)$-symmetric rings, were proposed and characterized in \cite{KUHH1,MW1,MW2,MW3,MW5,OH1,MW4}.

With the introduction of involution structures to rings,
Wang et al. \cite{WFW1} put forward the concept of $\ast$-symmetric rings.
A $\ast$-ring $R$ is said to be $\ast$-symmetric if $abc=0$ yields $acb^{\ast}=0$ for all $a, b, c\in R$.
It is known that every $\ast$-symmetric ring is symmetric,
and preliminary equivalent characterizations of $\ast$-symmetric rings were established in \cite{WFW1}.
However, the equivalence between $\ast$-reversible rings and $\ast$-symmetric rings has not been rigorously verified in existing literature, and there is a lack of systematic research on generalized $\ast$-reversible rings constrained by projections and nilpotents.

Motivated by the characterization of reversible rings in terms of idempotents and projections,
we first define pro-$\ast$-reversible rings as a projection-based generalization of $\ast$-reversible rings:
a $\ast$-ring $R$ is pro-$\ast$-reversible if $ab\in P(R)$ implies $b^{\ast}a\in P(R)$ for all $a,b\in R$,
where $P(R)$ denotes the set of all projections of $R$.
Meanwhile, by restricting the zero-divisor condition to nilpotents,
we introduce nil-$\ast$-reversible rings, requiring that $ab=0$ yields both $ab^{\ast}=0$ and $b^{\ast}a=0$ whenever $a\in N(R)$ and $b\in R$, with $N(R)$ standing for the set of nilpotents of $R$.

In this paper, we systematically study the above three classes of $\ast$-rings.
First, we give several new equivalent characterizations of reversible rings and $\ast$-reversible rings,
and prove that the class of $\ast$-reversible rings coincides exactly with the class of $\ast$-symmetric rings.
Second, we investigate the basic properties of pro-$\ast$-reversible rings,
clarify the implications among pro-$\ast$-reversible rings, $\ast$-reversible rings and reversible rings,
and construct concrete examples to distinguish these ring classes.
Finally, we study nil-$\ast$-reversible rings, which are defined by nilpotent elements.
Some characterizations of nil-$\ast$-reversible rings are obtained, and examples are constructed to explore their relationships with related ring classes.

Throughout the paper, all rings are assumed to be associative rings with unity unless otherwise specified.
For a ring $R$, we write $E(R)$ for the set of idempotent elements, $P(R)$ for the set of projections,
$N(R)$ for the set of nilpotent elements, and $C(R)$ for the center of $R$.

\section{$\ast$-reversible and $\ast$-symmetric rings}

We begin with the definitions of several related ring classes.
Let $R$ be a ring with unity.
A ring $R$ is said to be reversible if $ab=0$ implies $ba=0$ for all $a,b\in R$ (equivalently, $l(b)=r(b)$).
A ring equipped with an involution $\ast$ is called a $\ast$-ring. A $\ast$-ring $R$ is said to be $\ast$-reversible if $ab=0$ implies $b^{\ast}a=0$ for all $a,b\in R$.
We first establish a basic result concerning $\ast$-reversible rings.

\begin{lemma}\label{lem2.01}
If $R$ is a $\ast$-reversible ring, then $R$ is reversible.
\end{lemma}
\begin{proof}
Suppose $ab=0$. By the $\ast$-reversibility of $R$, we have $b^{\ast}a=0$.
Applying the $\ast$-reversible condition to $b^{\ast}a=0$ again yields $a^{\ast}b^{\ast}=0$, which implies $ba=0$.
\end{proof}

\begin{lemma}\label{lem2.02}
Let $R$ be a $\ast$-ring. For any $a \in R$, the following conditions are equivalent:

(1)  $R$ is $\ast$-reversible.

(2) $ab=0$ implies $b^{\ast}a=0$ for any $b\in R$.

(3) $ab=0$ implies $ab^{\ast}=0$ for any $b\in R$.

(4)  $l(a)=l(a^{\ast})$.

(5) $l(a) = r(a^{\ast})$.

(6)  $r(a)=r(a^{\ast})$.

(7) $r(a)= l(a^{\ast})$.
\end{lemma}

\begin{proof}
$(1) \Leftrightarrow (2)$. This is immediate from the definition of $\ast$-reversible rings.

$(2) \Rightarrow (3)$. The assertion follows directly from Lemma \ref{lem2.01}.

$(3) \Rightarrow (4)$.
Let $x\in l(a)$, so $xa=0$. By condition (3), we have $xa^{\ast}=0$, which implies $l(a)\subseteq l(a^{\ast})$.
Conversely, take $y\in l(a^{\ast})$, so $ya^{\ast}=0$. Again by (3), $ya=0$, so $l(a^{\ast})\subseteq l(a)$.
We therefore conclude $l(a)=l(a^{\ast})$.

$(4) \Rightarrow (5)$.
Take $x\in l(a)$, so $xa=0$. Taking involution yields $a^{\ast}x^{\ast}=0$, i.e., $a^{\ast}\in l(x^{\ast})$.
By (4), $l(x)=l(x^{\ast})$, so $a^{\ast}\in l(x)$, i.e., $a^{\ast}x=0$.
This yields $x\in r(a^{\ast})$, and thus $l(a)\subseteq r(a^{\ast})$.

Conversely, let $y\in r(a^{\ast})$, so $a^{\ast}y=0$. Then $a^{\ast}\in l(y)$.
By (4), $l(y)=l(y^{\ast})$, so $a^{\ast}y^{\ast}=0$. Taking involution gives $ya=0$, so $y\in l(a)$.
Hence $r(a^{\ast})\subseteq l(a)$, and $l(a)=r(a^{\ast})$.

$(5) \Rightarrow (6)$.
Let $x\in r(a)$, so $ax=0$. Then $a\in l(x)$.
By (5), $l(x)=r(x^{\ast})$, so $a\in r(x^{\ast})$, which means $x^{\ast}a=0$.
Taking involution yields $a^{\ast}x=0$, so $x\in r(a^{\ast})$. Thus $r(a)\subseteq r(a^{\ast})$.

Conversely, take $y\in r(a^{\ast})$, so $a^{\ast}y=0$. Then $a^{\ast}\in l(y)$.
By (5), $l(y)=r(y^{\ast})$, so $a^{\ast}\in r(y^{\ast})$, i.e., $y^{\ast}a^{\ast}=0$.
Taking involution gives $ay=0$, so $y\in r(a)$.
Hence $r(a^{\ast})\subseteq r(a)$, and $r(a)=r(a^{\ast})$.

$(6) \Rightarrow (7)$.
Let $x\in r(a)$, so $ax=0$. Taking involution gives $x^{\ast}a^{\ast}=0$, hence $a^{\ast}\in r(x^{\ast})$.
By (6), $r(x^{\ast})=r(x)$, so $a^{\ast}\in r(x)$, which implies $xa^{\ast}=0$, i.e., $x\in l(a^{\ast})$.
Thus $r(a)\subseteq l(a^{\ast})$.

Conversely, let $y\in l(a^{\ast})$, so $ya^{\ast}=0$. Then $a^{\ast}\in r(y)$.
By (6), $r(y)=r(y^{\ast})$, so $a^{\ast}\in r(y^{\ast})$, i.e., $y^{\ast}a^{\ast}=0$.
Taking involution yields $ay=0$, so $y\in r(a)$.
Hence $l(a^{\ast})\subseteq r(a)$, and therefore $r(a)=l(a^{\ast})$.

$(7) \Rightarrow (1)$.
Suppose $ab=0$. Then $a\in l(b)$.
By condition (7), $l(b)=r(b^{\ast})$, so $a\in r(b^{\ast})$, which means $b^{\ast}a=0$.
Therefore $R$ is $\ast$-reversible.
The proof is completed.
\end{proof}

We remark that conditions (5) and (7) in Lemma \ref{lem2.02} are equivalent.
In \cite{WFW1}, the authors introduced the notion of $\ast$-symmetric rings.
Recall that a ring $R$ is called $\ast$-symmetric if $abc=0$ implies $acb^{\ast}=0$ for all $a, b, c\in R$.
We will show that the class of $\ast$-symmetric rings coincides exactly with the class of $\ast$-reversible rings.
To do this, we first prove that $\ast$-symmetry is equivalent to one of the equivalent conditions listed in Lemma \ref{lem2.02}.
We begin with a known characterization of $\ast$-symmetric rings from the literature.

\begin{lemma}\label{lem2.03}\cite[Proposition 3.3]{WFW1}
A ring $R$ is $\ast$-symmetric if and only if for any $x_{1}, x_{2}, x_{3} \in R$,
$x_{1}x_{2}x_{3} = 0$, implies any of them works:
\[
x_{\sigma(1)}^{\ast}x_{\sigma(2)}x_{\sigma(3)} = 0,\quad
x_{\sigma(1)}x_{\sigma(2)}^{\ast}x_{\sigma(3)} = 0,\quad
x_{\sigma(1)}x_{\sigma(2)}x_{\sigma(3)}^{\ast} = 0
\]
within $\sigma\in S_3$, where $S_3$ denotes the symmetric group of order $3$.
\end{lemma}

\begin{theorem}\label{thm2.04}
Let $R$ be a $\ast$-ring. The following two statements are equivalent for all $a\in R$:

(1) $R$ is $\ast$-symmetric.

(2) $r(a)=l(a^{\ast})$.
\end{theorem}

\begin{proof}
$(1) \Rightarrow (2)$.
Suppose $ab = 0$. Rewrite it as $ab\cdot 1 = 0$. Since $R$ is $\ast$-symmetric, we have $ab^{\ast} = 0$, which yields $ba^{\ast} = 0$.
Consequently, $r(a)\subseteq l(a^{\ast})$.

$(2) \Rightarrow (1)$.
Let $xyz= 0$. Then $z\in r(xy)$.
By condition (2), $r(xy) \subseteq l((xy)^{\ast})$, so $z(xy)^{\ast} = 0$,
which yields $xyz^{\ast}=0$.
By Lemma \ref{lem2.03}, $R$ is $\ast$-symmetric.
\end{proof}

\begin{remark}
In Lemma \ref{lem2.02} and Theorem \ref{thm2.04}, every equality relation between annihilators can be weakened to an inclusion relation.
\end{remark}

We now construct an explicit example to demonstrate that a reversible ring need not be $\ast$-reversible.

\begin{lemma}\label{lem2.06}
Let $R$ be a $\ast$-ring. Define an involution $\ast$ on the ring $V_2(R)=\left\{\begin{pmatrix}a & b \\ 0 & a\end{pmatrix}\mid a,b\in R\right\}$ by
\[
\begin{pmatrix}
a & b \\
0 & a
\end{pmatrix}^{\ast}
=
\begin{pmatrix}
a^{\ast} & -b^{\ast} \\
0 & a^{\ast}
\end{pmatrix}.
\]
Then $V_2(R)$ forms a $\ast$-ring.
\end{lemma}
\begin{proof}
The result follows by directly verifying the axioms of an involution.
\end{proof}

\begin{example}\label{exa2.07}
Let $R$ be a reduced ring which is not $\ast$-reversible. Such an example is constructed in \cite[Example 2.3]{WFW1}.
By Lemma \ref{lem2.02}, there exist elements $a,b\in R$ satisfying $ab=0$, $b^{\ast}a\neq 0$ and $ab^{\ast}\neq 0$.

Let $V_2(R)$ be the ring defined in Lemma \ref{lem2.06}. According to \cite[Lemma 2.1]{WFW1}, $V_2(R)$ is symmetric. Since every symmetric ring is reversible, $V_2(R)$ is reversible.
Now take $A,B\in V_2(R)$ of the form
\[
A=\begin{pmatrix}
0 & a \\
0 & 0
\end{pmatrix},
\qquad
B=\begin{pmatrix}
b & c \\
0 & b
\end{pmatrix}
\]
for some $c\in R$. It is straightforward to verify that $AB=0$.

Compute the product $B^{\ast}A$:
\[
B^{\ast}A
=
\begin{pmatrix}
b^{\ast} & -c^{\ast} \\
0 & b^{\ast}
\end{pmatrix}
\begin{pmatrix}
0 & a \\
0 & 0
\end{pmatrix}
=
\begin{pmatrix}
0 & b^{\ast}a \\
0 & 0
\end{pmatrix}
\neq O_2.
\]
This shows that $V_2(R)$ fails to be $\ast$-reversible.
\end{example}

\section{$\ast$-reversible and pro-$\ast$-reversible rings}

It is well known that a ring $R$ is reversible if and only if $ab\in E(R)$ implies $ba\in E(R)$ for all $a,b\in R$.
We verify this equivalence as follows. First, suppose $R$ is reversible and $ab\in E(R)$.
Then $ab(1-ab)=0$, which yields $b(1-ab)a=0$ by the reversibility of $R$.
It follows that $ba\in E(R)$.
Conversely, let $ab=0$. Clearly $ab\in E(R)$, so the assumption gives $ba\in E(R)$.
Consequently, $ba=(ba)^2=baba=0$, which confirms $R$ is reversible.
Throughout this section, we derive several equivalent characterizations for generalized $\ast$-reversible rings.
Let $R$ be a $\ast$-ring. An element $p\in R$ is called a projection if $p=p^{2}=p^{\ast}$.
We write $P(R)$ for the set of all projections in $R$.
A ring $R$ is said to be abelian if $E(R) \subseteq C(R)$.

\begin{lemma}\label{lem3.01}
Let $R$ be a $\ast$-ring, and let $a,b\in R$. The following statements are equivalent:

(1) $R$ is reversible.

(2) $ab\in E(R)$ implies $ab=ba$.

(3) $ab\in P(R)$ implies $ab=ba$.

(4) $ab\in P(R)$ implies $ba\in P(R)$.
\end{lemma}

\begin{proof}
$(1)\Rightarrow(2)$.
Assume $ab\in E(R)$. As shown above, $ba\in E(R)$.
Since every reversible ring is abelian, all idempotents are central in $R$.
Thus $ab,ba\in C(R)$, and we obtain
\[
ba = b(ab)a = ab(ba) = (ab)^2 = ab.
\]

$(2)\Rightarrow(3)$.
Clearly, $ab\in P(R)$ implies $ab\in E(R)$.
By (2), we have $ab=ba$.

$(3)\Rightarrow(4)$.
If $ab\in P(R)$, then $ab=ba$ by condition (3).
It follows immediately that $ba\in P(R)$.

$(4)\Rightarrow(1)$.
If $ab=0$, then $ab\in P(R)$.
By condition (4), $ba\in P(R)$, so $ba=(ba)^2=baba=0$,
which proves that $R$ is reversible.
\end{proof}

Motivated by Lemma \ref{lem3.01}, we introduce the notion of pro-$\ast$-reversible rings as follows.

\begin{definition}
Let $R$ be a $\ast$-ring.
We say $R$ is pro-$\ast$-reversible if $ab\in P(R)$ implies $b^{\ast}a\in P(R)$ for all $a,b\in R$.
\end{definition}

\begin{lemma}\label{lem3.03}
Every pro-$\ast$-reversible ring is reversible.
\end{lemma}

\begin{proof}
Let $a,b\in R$ satisfy $ab\in P(R)$. Since $R$ is pro-$\ast$-reversible, we have $b^{\ast}a\in P(R)$.
Applying the pro-$\ast$-reversible condition again,
we further obtain $a^{\ast}b^{\ast}\in P(R)$, which implies $ba\in P(R)$.
By Lemma \ref{lem3.01}, $R$ is reversible.
\end{proof}

Recall that for any $\ast$-reversible ring, $ab=0$ implies $ab^{\ast}=b^{\ast}a=0$.
Therefore, Lemma \ref{lem3.03} immediately yields the following corollary.

\begin{corollary}\label{cor3.04}
Let $R$ be a $\ast$-ring. The following conditions are equivalent:

(1) $R$ is pro-$\ast$-reversible.

(2) $ab\in P(R)$ implies $ab^{\ast}=b^{\ast}a\in P(R)$ for all $a,b\in R$.
\end{corollary}

\begin{proof}
It suffices to prove $(1)\Rightarrow(2)$.
Suppose $ab\in P(R)$. By (1), $b^{\ast}a\in P(R)$.
Combining Lemma \ref{lem3.01} and Lemma \ref{lem3.03}, we obtain $ab^{\ast}=b^{\ast}a$,
so that $ab^{\ast}=b^{\ast}a\in P(R)$.
\end{proof}

\begin{proposition}\label{prop3.05}
Every pro-$\ast$-reversible ring is $\ast$-reversible.
\end{proposition}

\begin{proof}
Suppose $ab=0$. Then clearly $ab\in P(R)$.
By Corollary \ref{cor3.04}, $ab^{\ast}=b^{\ast}a$.
We compute
\[
b^{\ast}a a^{\ast} = ab^{\ast}a^{\ast} = a(ab)^{\ast}=0.
\]
This yields $(b^{\ast}a)(b^{\ast}a)^{\ast}=0$.
Since $b^{\ast}a\in P(R)$, we conclude $b^{\ast}a=0$.
This shows that $R$ is $\ast$-reversible.
\end{proof}

Motivated by Corollary \ref{cor3.04} and Proposition \ref{prop3.05}, we pose the following two conjectures.

\begin{question}\label{Q3.06}
Let $R$ be a $\ast$-ring, and let $a,b\in R$. Are the following conditions equivalent?

(1) $R$ is pro-$\ast$-reversible.

(2) $ab\in P(R)$ implies $ab^{\ast}\in P(R)$.

(3) $ab\in P(R)$ implies $ab^{\ast}=b^{\ast}a$.
\end{question}

\begin{question}\label{Q3.07}
Is every $\ast$-reversible ring pro-$\ast$-reversible?
\end{question}

Regarding Question \ref{Q3.06}, we verify that the implications $(1)\Rightarrow(2)$ and $(2)\Rightarrow(3)$ hold.
Indeed, $(1) \Rightarrow (2)$ is immediate from Corollary \ref{cor3.04}.
As for $(2) \Rightarrow (3)$, we present the following reasoning.

\begin{lemma}\label{lem3.08}
Let $R$ be a $\ast$-ring satisfying that $ab\in P(R)$ implies
$ab^{\ast}\in P(R)$ for all $a,b\in R$. Then $R$ is abelian.
\end{lemma}
\begin{proof}
Let $e\in E(R)$. Then $(1-e)e=0\in P(R)$.
By the given hypothesis, $(1-e)e^{\ast}\in P(R)$.
Hence $[(1-e)e^{\ast}]^{\ast}=(1-e)e^{\ast}$, which yields $e=e^{\ast}$.
Therefore $E(R)\subseteq P(R)$.
It then follows from \cite[Proposition 3.4]{WFW1} that $R$ is abelian.
\end{proof}

\begin{theorem}\label{thm3.09}
Let $R$ be a $\ast$-ring such that $ab\in P(R)$ implies $ab^{\ast}\in P(R)$ for all $a,b\in R$.
Then $R$ is $\ast$-reversible.
\end{theorem}

\begin{proof}
Suppose $ab=0$. Then the hypothesis implies $ab^{\ast}\in P(R)$.
By Lemma \ref{lem3.08}, $R$ is abelian, so $ab^{\ast}\in C(R)$.
We compute:
\[
a^{\ast}ab^{\ast}=ab^{\ast}a^{\ast}=a(ab)^{\ast}=0.
\]
Thus,
\[
(ab^{\ast})^{\ast}(ab^{\ast})=ba^{\ast}ab^{\ast}=0.
\]
As $ab^{\ast}$ is a projection, we obtain
\[
ab^{\ast}=(ab^{\ast})^{2}=(ab^{\ast})^{\ast}(ab^{\ast})=0.
\]
This shows that $R$ is $\ast$-reversible, as required.
\end{proof}

\begin{corollary}\label{cor3.10}
Let $R$ be a $\ast$-ring. If $ab\in P(R)$ implies $ab^{\ast}\in P(R)$ for all $a,b\in R$, then $ab^{\ast}=b^{\ast}a$.
\end{corollary}
\begin{proof}
By Theorem \ref{thm3.09}, $R$ is $\ast$-reversible under the present hypotheses.
In particular, $R$ is reversible.
As $ab^{\ast}\in P(R)$, Lemma \ref{lem3.01} forces $ab^{\ast}=b^{\ast}a$.
\end{proof}

The following corollary establishes the equivalence between conditions (1) and (2) listed in Question \ref{Q3.06}.

\begin{corollary}\label{cor3.11}
Let $R$ be a $\ast$-ring, and let $a,b\in R$. The following statements are equivalent:

(1) $R$ is pro-$\ast$-reversible.

(2) $ab\in P(R)$ implies $ab^{\ast}\in P(R)$.
\end{corollary}
\begin{proof}
It suffices to prove the implication $(2)\Rightarrow(1)$.
Assume condition (2): if $ab\in P(R)$, then $ab^{\ast}\in P(R)$.
By Theorem \ref{thm3.09} and Lemma \ref{lem3.01}, $R$ is $\ast$-reversible,
then $R$ is reversible, which yields $ab^{\ast}=b^{\ast}a\in P(R)$.
Hence $R$ is pro-$\ast$-reversible, completing the proof.
\end{proof}

\begin{remark}\label{rem3.12}
In general, the implication $(3) \Rightarrow (1)$ in Question \ref{Q3.06} does not hold.
Let $R = \mathbb{Z}_{6} \oplus \mathbb{Z}_{6}$, with involution defined by $\alpha^{\ast} = (y, x)$ for $\alpha = (x, y) \in R$.
According to \cite[Example 2.3]{WFW1}, $R$ is not $\ast$-reversible, and hence not pro-$\ast$-reversible by Proposition \ref{prop3.05}.
Nevertheless, $R$ is commutative, so the condition $ab\in P(R)\Rightarrow ab^{\ast}=b^{\ast}a$
is automatically satisfied for all $a,b\in R$.
\end{remark}

We now address Question \ref{Q3.07}. We will construct an explicit example to show that a $\ast$-reversible ring need not be pro-$\ast$-reversible in general.

\begin{example}\label{exa3.16}
Take the ring $R=V_2(\mathbb{C})$ from \cite[Example 2.2]{WFW1}, where
\[
V_2(\mathbb{C})=\left\{
\begin{pmatrix}
a & b \\
0 & a
\end{pmatrix}
\;\bigg|\;
a,b\in\mathbb{C}
\right\}.
\]
Define an involution $\ast$ on $R$ by
\[
\begin{pmatrix}
a & b \\
0 & a
\end{pmatrix}^{\ast}
=
\begin{pmatrix}
a & -b \\
0 & a
\end{pmatrix}.
\]
It is known that $R$ is $\ast$-reversible (recall that the classes of $\ast$-reversible and $\ast$-symmetric rings coincide). Set
\[
a=\begin{pmatrix}
1 & 1 \\
0 & 1
\end{pmatrix},
\qquad
b=\begin{pmatrix}
1 & -1 \\
0 & 1
\end{pmatrix}.
\]
Direct computation yields $ab=E_2\in P(R)$, while
\[
ab^{\ast}=
\begin{pmatrix}
1 & 2 \\
0 & 1
\end{pmatrix}
\notin P(R).
\]
Note that $P(R)=\{O_2,E_2\}$. Consequently, $R$ is not pro-$\ast$-reversible.
\end{example}

\begin{remark}\label{rem3.14}
Summarizing the above results, we record the following interrelations between these ring classes:
\begin{enumerate}
    \item Every pro-$\ast$-reversible ring is $\ast$-reversible (Proposition \ref{prop3.05}), but the converse fails in general (Example \ref{exa3.16}).
    \item Every pro-$\ast$-reversible ring satisfies condition (3) of Question \ref{Q3.06}, but the converse fails in general (Remark \ref{rem3.12}).
    \item Condition (3) of Question \ref{Q3.06} does not guarantee that $R$ is $\ast$-reversible (Remark \ref{rem3.12}).
\end{enumerate}
\end{remark}

We next investigate further consequences of condition (3) from Question \ref{Q3.06}.

\begin{lemma}\label{lem3.15}
Suppose that $ab\in P(R)$ implies $b^{\ast}a=ab^{\ast}$ for all $a\in N(R)$ and $b\in R$. Then $ee^{\ast}=e^{\ast}e$ for all $e\in E(R)$.
\end{lemma}
\begin{proof}
For any idempotent $e\in E(R)$ and $r\in R$, we have $er(1-e)e=0\in P(R)$. By the hypothesis,
\begin{equation}\label{eq:2.1}
er(1-e)e^{\ast}=e^{\ast}er(1-e).
\end{equation}
Substitute $r=e^{\ast}$ into \eqref{eq:2.1}. This gives
\[
ee^{\ast}(1-e)e^{\ast}=e^{\ast}ee^{\ast}(1-e).
\]
Expanding both sides yields
\[
ee^{\ast}-ee^{\ast}ee^{\ast}=e^{\ast}ee^{\ast}-e^{\ast}ee^{\ast}e.
\]
Rearranging terms, we obtain
\begin{equation}\label{eq:2.2}
ee^{\ast}-e^{\ast}ee^{\ast}=ee^{\ast}ee^{\ast}-e^{\ast}ee^{\ast}e.
\end{equation}
Taking involution of both sides of \eqref{eq:2.2}, we see that $ee^{\ast}-e^{\ast}ee^{\ast}$ is self-adjoint. This forces the identities $ee^{\ast}e=e^{\ast}ee^{\ast}$ and $ee^{\ast}=e^{\ast}ee^{\ast}$.

On the other hand, observe that $e^{\ast}r(1-e)^{\ast}e^{\ast}=0$. Applying the hypothesis, we get
\[
e^{\ast}r(1-e)^{\ast}e=ee^{\ast}r(1-e)^{\ast}.
\]
Set $r=e$ in the above equality:
\[
e^{\ast}e(1-e)^{\ast}e=ee^{\ast}e(1-e)^{\ast}.
\]
Expanding gives
\[
e^{\ast}e-e^{\ast}ee^{\ast}e=ee^{\ast}e-ee^{\ast}ee^{\ast},
\]
which simplifies to $e^{\ast}e=ee^{\ast}e$. Combining this with $ee^{\ast}=e^{\ast}ee^{\ast}$, we conclude $ee^{\ast}=e^{\ast}e$.
\end{proof}

\begin{proposition}\label{prop3.16}
Let $R$ be a $\ast$-ring such that $ab\in P(R)$ implies $b^{\ast}a=ab^{\ast}$ for all $a\in N(R)$ and $b\in R$. Then $R$ is abelian.
\end{proposition}
\begin{proof}
By Lemma \ref{lem3.15}, take any $e\in E(R)$ and $r\in R$.
Then $(1-e)^{\ast}er(1-e)\in N(R)$, and $(1-e)^{\ast}er(1-e)e=0$. The hypothesis implies
\[
(1-e)^{\ast}er(1-e)e^{\ast}=e^{\ast}(1-e)^{\ast}er(1-e)=0.
\]
Applying the hypothesis again to the equality $(1-e)^{\ast}er(1-e)e^{\ast}=0$, we obtain
\[
(1-e)^{\ast}er(1-e)e=e(1-e)^{\ast}er(1-e)=0.
\]
From Lemma \ref{lem3.15}, $e(1-e)^{\ast}er(1-e)=0$, so $(1-e)^{\ast}er(1-e)=0$.

Now set $A=\big[er(1-e)\big]^{\ast}$ and $B=1-e$. Then $AB=0$, and the assumed identity yields $AB^{\ast}=B^{\ast}A$, i.e.,
\[
\big[er(1-e)\big]^{\ast}(1-e)^{\ast}=(1-e)^{\ast}\big[er(1-e)\big]^{\ast}.
\]
This implies $er(1-e)=0$ for all $r\in R$, hence $E(R) \subseteq C(R)$, and $R$ is abelian.
\end{proof}

\begin{corollary}\label{cor3.17}
Suppose $ab\in P(R)$ implies $b^{\ast}a=ab^{\ast}$ for all $a,b\in R$. Then $R$ is abelian.
\end{corollary}

\section{$\ast$-reversible and nil-$\ast$-reversible rings}

In this section, we further study generalized $\ast$-reversible rings.
Inspired by Lemma \ref{lem2.01},
we focus on $a\in N(R)$ and introduce the definition of nil-$\ast$-reversible rings below.

\begin{definition}\label{def4.01}
Let $R$ be a $\ast$-ring. We say $R$ is nil-$\ast$-reversible
if $ab=0$ implies $ab^{\ast}=b^{\ast}a=0$ for all $a\in N(R)$ and $b\in R$.
\end{definition}

\begin{lemma}\label{lem4.02}
Let $R$ be a nil-$\ast$-reversible ring. Then $l(a)=r(a)$ holds for every $a\in N(R)$.
\end{lemma}
\begin{proof}
For any $r\in l(a)$. Then $a^{\ast}r^{\ast}=0$.
Since $a^{\ast}\in N(R)$ and $R$ is nil-$\ast$-reversible,
the hypothesis gives $a^{\ast}r=0$.
Applying the nil-$\ast$-reversible condition again yields $r^{\ast}a^{\ast}=0$,
which means $ar=0$. Thus $l(a)\subseteq r(a)$.

Conversely, let $r\in r(a)$, so $ar=0$. By the nil-$\ast$-reversible property, $ar^{\ast}=0$. Applying the same condition again, we get $ra=0$, hence $r(a)\subseteq l(a)$. Combining the two inclusions, we obtain $l(a)=r(a)$.
\end{proof}

Recall that a ring $R$ is nil reversible if $ab=0$ implies $ba=0$ for all $a\in N(R)$ and $b\in R$. It follows immediately from Lemma \ref{lem4.02} that every nil-$\ast$-reversible ring is nil reversible. The example below demonstrates that the reverse implication fails in general.

\begin{example}\label{ex4.03}
Let $S$ be a reduced ring which is not $\ast$-reversible. Then there exist $a,b\in S$ satisfying $ab=0$ and $ab^{\ast}\neq 0$. Consider the ring
\[
R=V_{2}(S)
=\left\{
\begin{pmatrix}
x & y \\
0 & x
\end{pmatrix}
\;\bigg|\;
x, y\in S
\right\}
\]
with involution $\ast$ defined by
\[
\begin{pmatrix}
x & y \\
0 & x
\end{pmatrix}^{\ast}
=
\begin{pmatrix}
x^{\ast} & -y^{\ast} \\
0 & x^{\ast}
\end{pmatrix}.
\]
The ring $V_2(S)$ is nil reversible.
In fact, $V_2(S)$ is reversible,
and every reversible ring is nil reversible. Now set
\[
A=\begin{pmatrix}
0 & a \\
0 & 0
\end{pmatrix},
\qquad
B=\begin{pmatrix}
b & c \\
0 & b
\end{pmatrix}\in V_2(S)
\]
for some $c\in S$. Direct calculation shows $AB=O_2$, while
\[
AB^{\ast}=
\begin{pmatrix}
0 & ab^{\ast} \\
0 & 0
\end{pmatrix}\neq O_2.
\]
Accordingly, $V_2(S)$ is not nil-$\ast$-reversible.
\end{example}

\begin{theorem}\label{thm4.04}
Let $R$ be a $\ast$-ring, and let $a\in N(R)$, $b\in R$. The following statements are equivalent:

(1) $R$ is nil $\ast$-reversible.

(2) $l(a)=l(a^{\ast})=r(a)$.
\end{theorem}

\begin{proof}
$(1)\Rightarrow(2)$. By Lemma \ref{lem4.02}, it suffices to show $l(a^{\ast}) = r(a)$.
Take arbitrary $x\in r(a)$. Since $ax=0$, condition (1) yields $ax^{\ast}=0$, and then $xa^{\ast}=0$.
Thus $r(a) \subseteq l(a^{\ast})$.

Now take any $y\in l(a^{\ast})$, that is, $ay^{\ast}=0$. Then by (1) we get $ay=0$.
Hence $y\in r(a)$, so $l(a^{\ast}) \subseteq r(a)$.
Combining the two inclusions, we obtain $l(a^{\ast}) = r(a)$.

$(2)\Rightarrow(1)$. Suppose $ab=0$. Since $l(a^{\ast})=r(a)$, we have $b\in l(a^{\ast})$, i.e., $ba^{\ast}=0$, which gives $ab^{\ast}=0$.
In addition, $l(a)=r(a)$ implies that $R$ is nil-reversible.
Together with $ab^{\ast}=0$, we conclude $b^{\ast}a = ab^{\ast}=0$.
\end{proof}

The following proposition provides several characterizations of nil-$\ast$-reversible rings,
analogous to those for reversible rings (see Lemma \ref{lem3.01}).

\begin{lemma}\label{lem4.05}
Every nil-reversible ring is abelian.
\end{lemma}
\begin{proof}
Let $R$ be a nil-reversible ring and $e\in E(R)$.
For any $r\in R$, we have $er(1-e)e=0$.
Since $R$ is nil-reversible, the defining condition yields $e\cdot er(1-e)=0$, i.e., $er(1-e)=0$.
Then, we conclude that $R$ is abelian.
\end{proof}

\begin{proposition}\label{prop4.06}
Let $R$ be a $\ast$-ring, $a\in N(R)$ and $b\in R$. The following statements are equivalent:

(1) $R$ is nil $\ast$-reversible.

(2) $ab \in E(R)$ implies that $ab^{\ast}=b^{\ast}a \in E(R)$.

(3) $ab \in P(R)$ implies that $ab^{\ast}=b^{\ast}a \in P(R)$.
\end{proposition}
\begin{proof}
$(1)\Rightarrow(2)$. By Lemma \ref{lem4.05}, $R$ is abelian.
Suppose $ab \in E(R)$. Since $a\in N(R)$, there exists a positive integer $n$ such that $a^n=0$, which gives $ab=a^n b^n=0$, it is because $ab\in C(R)$. Hence, $ab^{\ast}=b^{\ast}a=0\in E(R)$.

$(2)\Rightarrow(3)$. Let $ab \in P(R)$. Then $ab\in E(R)$.
By condition (2), $ab^{\ast}=b^{\ast}a$.
From Proposition \ref{prop3.16}, $R$ is abelian. Together with $ab^{\ast}\in E(R)$, we have $ab^{\ast}\in C(R)$.
Since $a\in N(R)$, there exists $n\in \mathbb{N}^+$ satisfying
\[
ab^{\ast}=(ab^{\ast})^n=a^n(b^{\ast})^n=0\in P(R).
\]

$(3)\Rightarrow(1)$.
Suppose $ab=0$. Then $ab\in P(R)$, so condition (3) yields $ab^{\ast}=b^{\ast}a \in P(R)$.
We compute:
\[
b^{\ast}aa^{\ast}=ab^{\ast}a^{\ast}=a(ab)^{\ast}=0.
\]
As $b^{\ast}a \in P(R)$, we can see that
\[
b^{\ast}a=(b^{\ast}a)(b^{\ast}a)=(b^{\ast}a)(b^{\ast}a)^{\ast}=b^{\ast}aa^{\ast}b=0.
\]
This proves $ab^{\ast}=b^{\ast}a=0$, so $R$ is nil-$\ast$-reversible.
\end{proof}

We now analyze several conditions involving nilpotents $a\in N(R)$, which are listed below:
\begin{align*}
(1)&\ ab\implies ab^{\ast}=b^{\ast}a;\\
(2)&\ ab\in P(R)\implies ab^{\ast}\in P(R);\\
(3)&\ ab\in P(R)\implies b^{\ast}a\in P(R).
\end{align*}
In fact, Proposition \ref{prop3.16} tells us that if $ab\in P(R)\implies ab^{\ast}=b^{\ast}a$ holds for all $a\in N(R)$,
then $R$ must be abelian.
Under this abelian setting, the condition $ab\in P(R)$ simplifies to $ab=0$.
Indeed, take any $a\in N(R)$, there exists a positive integer $n$ with $a^n=0$, so $ab=(ab)^n=a^n b^n=0$.

Nevertheless, the following example shows that a ring may satisfy $ab=0\implies ab^{\ast}=b^{\ast}a$ without being nil-$\ast$-reversible. In other words, $ab^{\ast}$ and $b^{\ast}a$ need not to be zero in general.

\begin{example}\label{exa4.07}
Let $R=\mathbb{Z}_9\oplus\mathbb{Z}_9$. Define an involution on $R$ by $\alpha^{\ast}=(y,x)$ for every $\alpha=(x,y)\in R$. Since $R$ is commutative, the implication $ab\in P(R)\implies ab^{\ast}=b^{\ast}a$ holds for all $a\in N(R)$.
Set $A=(3,0)$ and $B=(3,2)$. Clearly $A\in N(R)$ and $AB=0$, but $AB^{\ast}=(6,0)\neq 0$. Hence $R$ is not nil-$\ast$-reversible.
\end{example}

We next establish a characterization for nil-$\ast$-reversible rings.

\begin{proposition}\label{prop4.08}
Let $R$ be a $\ast$-ring, $a\in N(R)$ and $b\in R$. The following two statements are equivalent:

(1) $ab=0$ implies $b^{\ast}a=0$.

(2) $ab \in P(R)$ implies $b^{\ast}a \in P(R)$.
\end{proposition}
\begin{proof}
$(1)\Rightarrow(2)$. We first claim all projections of $R$ are central, i.e., $P(R)\subseteq C(R)$.
Take any $p\in P(R)$ and $r\in R$. We have $pr(1-p)p=0$. By condition (1), $p^{\ast}pr(1-p)=0$, which yields $pr(1-p)=0$.
Similarly, $(1-p)rp=0$, and condition (1) gives $(1-p)^{\ast}(1-p)rp=0$, so $(1-p)rp=0$. This forces $pr=rp$.

Now assume $ab\in P(R)$. Since $a\in N(R)$, there exists $n\in\mathbb{N}$ such that $a^n=0$, hence $ab=(ab)^n=a^n b^n=0$. Condition (1) then gives $b^{\ast}a=0\in P(R)$.

$(2)\Rightarrow(1)$.  We first verify $pr=rp$ for all $p\in P(R)$.
Let $p\in P(R)$ and $r\in R$. Note $pr(1-p)p=0\in P(R)$. By (2), $p^{\ast}pr(1-p)\in P(R)$, so $pr(1-p)\in P(R)$.
Observe $pr(1-p)\in N(R)$. There exists $m\in\mathbb{N}$ satisfying $\big[pr(1-p)\big]^m=0$, so $pr(1-p)=0$. By symmetry, $(1-p)rp=0$, which implies $pr=rp$.

Now suppose $ab=0$. Condition (2) guarantees $b^{\ast}a\in P(R)$. Since $a\in N(R)$, we have $a^n=0$ for some $n\in\mathbb{N}$. Thus
$b^{\ast}a=(b^{\ast}a)^n=(b^{\ast})^n a^n=0$.
\end{proof}

A natural question arises: if $ab=0$ yields $b^{\ast}a=0$ for all $a\in N(R),b\in R$, does it necessarily follow that $ab^{\ast}=0$? Equivalently, if a $\ast$-ring satisfies condition (1) of Proposition \ref{prop4.08}, is it nil-$\ast$-reversible? We construct an example to give a negative answer, and first prepare an auxiliary lemma.

\begin{lemma}\label{lem4.09}
Let $R$ be a $\ast$-ring, and let $a\in N(R)$, $b\in R$. The following statements are equivalent:

(1) $ab=0$ implies $b^{\ast}a=0$.

(2)  $r(a) = r(a^{\ast})$.

(3)  $l(a) = l(a^{\ast})$.
\end{lemma}

\begin{proof}
$(1) \Rightarrow (2)$.
Take arbitrary $x\in r(a)$, so $ax=0$.
By condition (1), we obtain $x^{\ast}a=0$.
Taking involution on both sides yields $a^{\ast}x=0$,
which means $x\in r(a^{\ast})$. Hence $r(a)\subseteq r(a^{\ast})$.

Conversely, for any $y\in r(a^{\ast})$, so $a^{\ast}y=0$.
Applying condition (1) to this equality gives $y^{\ast}a^{\ast}=0$, and then $ay=0$, so $y\in r(a)$.
Therefore $r(a^{\ast})\subseteq r(a)$.
Combining the two inclusions, we conclude $r(a)=r(a^{\ast})$.

$(2) \Rightarrow (3)$.
For any $x\in l(a)$, so $xa=0$.
Taking involution yields $a^{\ast}x^{\ast}=0$, i.e., $x^{\ast}\in r(a^{\ast})$.
By condition (2), $r(a^{\ast})=r(a)$, so $ax^{\ast}=0$.
Taking involution again, we obtain $xa^{\ast}=0$,
which implies $x\in l(a^{\ast})$. Thus $l(a)\subseteq l(a^{\ast})$.

Conversely, take arbitrary $y\in l(a^{\ast})$, so $ya^{\ast}=0$.
Taking involution gives $ay^{\ast}=0$, so $y^{\ast}\in r(a)$. By condition (2), $r(a)=r(a^{\ast})$, hence $a^{\ast}y^{\ast}=0$. Taking involution yields $ya=0$, which means $y\in l(a)$. Therefore $l(a^{\ast})\subseteq l(a)$.
Together we have $l(a)=l(a^{\ast})$.

$(3) \Rightarrow (1)$.
Suppose $ab=0$. Taking involution gives $b^{\ast}a^{\ast}=0$, so $b^{\ast}\in l(a^{\ast})$.
By condition (3), $l(a^{\ast})=l(a)$, which implies $b^{\ast}\in l(a)$, i.e., $b^{\ast}a=0$.
\end{proof}

\begin{example}\label{exa4.10}
Let $R=T_2(\mathbb{C})$, equipped with the involution defined by
\[
\begin{pmatrix}
a & b \\
0 & c
\end{pmatrix}^{\ast}
=
\begin{pmatrix}
c & b \\
0 & a
\end{pmatrix}.
\]
By Lemma \ref{lem4.09}, $R$ satisfies condition (1) of Proposition \ref{prop4.08}, since $l(x)=l(x^{\ast})$ holds for all $x\in N(R)$.
Take $a=e_{12}\in N(R)$ and $b=e_{11}$. Direct computation gives $ab=O_2$, but $ba=e_{12}\neq 0$. This shows $R$ fails to be nil-reversible, so $R$ cannot be nil-$\ast$-reversible.
\end{example}

\begin{proposition}\label{prop4.11}
Let $R$ be a $\ast$-ring, and let $a\in N(R)$, $b\in R$. The following statements are equivalent:

(1) $ab=0$ implies $ab^{\ast}=0$.

(3) $l(a) = r(a^{\ast})$.
\end{proposition}
\begin{proof}
$(1) \Rightarrow (2)$.
For any $x\in l(a)$, so $xa=0$. Taking involution yields $a^{\ast}x^{\ast}=0$.
By condition (1), we get $a^{\ast}x=0$, which means $x\in r(a^{\ast})$. Thus $l(a)\subseteq r(a^{\ast})$.

Conversely, take arbitrary $y\in r(a^{\ast})$, so $a^{\ast}y=0$.
Applying condition (1) to this identity gives $a^{\ast}y^{\ast}=0$. Taking involution, we obtain $ya=0$, so $y\in l(a)$.
Therefore $r(a^{\ast})\subseteq l(a)$.
Combining the two inclusions, we conclude $l(a)=r(a^{\ast})$.

$(2) \Rightarrow (1)$.
Suppose $ab=0$. Taking involution on both sides yields $b^{\ast}a^{\ast}=0$, so $b^{\ast}\in l(a^{\ast})$.
By condition (2), $l(a^{\ast})=r(a)$, hence $b^{\ast}\in r(a)$, i.e., $ab^{\ast}=0$.
By definition, $R$ is left nil $\ast$-reversible.
\end{proof}

We now give a proof for the following implication:
\begin{center}
$ab=0$ implies $ab^{\ast}=0$ $\quad\Longrightarrow\quad$ $ab\in P(R)$ implies $ab^{\ast}\in P(R)$.
\end{center}
At present, it remains unknown whether the converse implication holds.

\begin{proposition}\label{prop4.12}
Let $p\in P(R)$ and $r_1,r_2\in R$. Then
\[
\bigl(ab=0\implies ab^{\ast}=0\bigr) \Longrightarrow \bigl(ab\in P(R)\implies ab^{\ast}\in P(R)\bigr).
\]
\end{proposition}

\begin{proof}
Take any $p\in P(R)$ and $r_1,r_2\in R$. We have
\[
pr_1(1-p)pr_2^{\ast}=0.
\]
Applying the hypothesis yields
\[
pr_1(1-p)\bigl[pr_2^{\ast}\bigr]^{\ast}=pr_1(1-p)r_2 p=0,
\]
which implies $pr_1 r_2 p = pr_1 p r_2 p$.

Now suppose $ab\in P(R)$. We compute:
\[
ab=(ab)ab(ab)=(ab)a(ab)b(ab)=aba^2 b^2 ab.
\]
Using the identity $pr_1 r_2 p = pr_1 p r_2 p$ repeatedly, we get
\[
aba^2 b^2 ab=(ab)a^2 b^2(ab)=(ab)a^2(ab)b^2(ab)=aba^3 b^3 ab.
\]
Iterating this procedure, we eventually obtain
\[
ab=aba^n b^n ab=0,
\]
since $a\in N(R)$.
Invoking the hypothesis once again, we conclude $ab^{\ast}=0\in P(R)$.
This completes the proof.
\end{proof}

\begin{corollary}\label{cor4.13}
Let $R$ be a $\ast$-ring, and let $a\in N(R)$, $b\in R$. The following statements are equivalent:

(1) $R$ is nil $\ast$-reversible.

(2) $ab \in P(R)$ implies that $ab^{\ast} \in P(R)$ and $b^{\ast}a \in P(R)$.
\end{corollary}
\begin{proof}
We only need to verify $(2)\Rightarrow(1)$. Assume condition (2) holds and take $a\in N(R),b\in R$ with $ab=0$. By Proposition \ref{prop4.08}, $b^{\ast}a=0$, and $P(R)\subseteq C(R)$.
From $b^{\ast}a=0$, we take involution to get $a^{\ast}b=0$. Applying hypothesis (2), $a^{\ast}b^{\ast}=ba\in P(R)$.
Since $P(R)\subseteq C(R)$ and $a\in N(R)$, there exists $n\in\mathbb{N}$ such that $ba=b^n a^n=0$.
This shows $R$ is nil-reversible. Combined with $b^{\ast}a=0$, we obtain $ab^{\ast}=0$, so $R$ is nil-$\ast$-reversible.
\end{proof}

\section*{Acknowledge}
The authors sincerely thank Prof. Junchao Wei for his valuable comments and suggestions,
which greatly improved the presentation of this paper.
This work is supported by the National Natural Science Foundation of China (12471133),
the Natural Science Foundation of Jiangsu Province (BK20200944, BK20220589),
and the Anhui Provincial Department of Education Natural Science Research Project (2025AHGXZK30855).


\begin{thebibliography}{s2}

\bibitem{KUHH1}  G. Kafkas, B. Ungor, S. Halicioglu and A. Harmanci,
Generalized symmetric rings,
Algebra Discrete Math. \textbf{12} (2011) 78-84.

\bibitem{L1} J. Lambek,
On the representation of modules by sheaves of factor modules,
Canad. Math. Bull. \textbf{14} (1971) 359-368.

\bibitem{M1} G. Marks,
Reversible and symmetric rings,
J. Pure Appl. Algebra \textbf{174} (2002) 311-318.

\bibitem{MW1} F. Meng and J. Wei,
Some properties of $e-$symmetric rings,
Turk. J. Math.  \textbf{42} (2018) 2389-2399.

\bibitem{MW2} F. Meng and J. Wei,
$(g, e)-$symmetric rings,
Algebra Colloq. \textbf{31} (2024) 263-270.

\bibitem{MW3} F. Meng and J. Wei,
$e-$symmetric rings,
Commun. Contemp. Math. \textbf{20} (2018) Paper No.1750039, 8 pp.

\bibitem{MW5}  F. Meng, J. Wei, and  R. Chen,
Weak $e-$symmetric rings,
Comm. Algebra \textbf{51} (2023) 3042-3050.

\bibitem{OH1} L. Ouyang and H. Chen,
On weak symmetric rings,
Comm. Algebra \textbf{38} (2010)  697-713.

\bibitem{MW4}  J. Wei,
Generalized weakly symmetric rings,
J. Pure Appl. Algebra \textbf{218} (2014) 1594-1603.

\bibitem{WFW1} X.  Wang, S. Fan and J. Wei,
$\ast$-symmetric rings,
J. Algebra Appl. \textbf{25} (2026) Paper No.2650098, 11 pp.
\end{thebibliography}
\end{document}